\documentclass[11pt,twoside]{article}

\usepackage{mathrsfs,amsfonts,amsmath,amssymb}
\usepackage{latexsym}
\usepackage{comment} 
\newtheorem{satz}{Theorem}

\newtheorem{theorem}[satz]{Theorem}
\newtheorem{lemma}[satz]{Lemma}

\newtheorem{corollary}[satz]{Corollary}
\newtheorem{remark}[satz]{Remark}

\def\T{\mathsf{T}}

\def\Z{\mathbb {Z}}
\def\F{\mathbb {F}}
\def\E{\mathsf{E}}

\def\I{{\cal I}}
\def\a{\alpha}

\def\C{\mathbb{C}}

\def\d{\delta}
\def\o{\omega}
\def\({\big (}
\def\){\big )}

\def\le{\leqslant}
\def\ge{\geqslant}
\def\_phi{\varphi}
\def\eps{\varepsilon}

\def\Gr{{\mathbf G}}

\def\Cf{{\mathcal C}}
\def\la{\lambda}
\def\D{\Delta}

\def\R{\mathbb {R}}

\author{Shkredov I.D.}
\title{On multiplicative properties of combinatorial cubes
	\footnote{This work is supported by the Russian Science Foundation under grant 19--11--00001.}
}
\date{}
\begin{document}
	\maketitle

	\begin{abstract}
	We obtain a series of
	lower bounds for the product set of combinatorial cubes, as well as some non--trivial upper estimates for the multiplicative energy of such sets.  
	\end{abstract}

\section{Introduction}

The notion of a combinatorial (Hilbert) cube in $\R$ was defined by Hilbert in \cite{Hilbert} as follows: having a set of non--zero integers $a_0, a_1,\dots,a_d$ put 
\begin{equation}\label{def:cube_intr}
	Q(a_0,a_1,\dots, a_d) =  \left\{ a_0 + \sum_{j=1}^d \eps_j a_j ~:~ \eps_j \in \{0,1\} \right\} \,.
\end{equation}
Combinatorial cubes play an important role in the proof of Szemer\'edi's celebrated theorem \cite{Sz}. 
There is a wide literature on Hilbert cubes, e.g., see 
\cite{Csikvari}---\cite{H_P}
and other papers.

One can see from definition \eqref{def:cube_intr} that any combinatorial cube is
an additively rich set.
If so, then by the sum--product phenomenon (see, e.g., \cite{TV}) one can suppose that the cubes should have relatively weak multiplicative structure. 
This idea  was introduced in  \cite{H_mult}, where the first results on cubes in the prime field $\F_p$ were obtained. 
The bounds here depended on the characteristic $p$ (e.g., see \cite[Proposition 3.1]{H_mult}). 
Let us formulate some particular cases of the main results of our paper, see Theorems \ref{t:energy_Q1}, \ref{t:energy_Q1_m}, \ref{t:sigma_new} 
below.

\begin{theorem}
	Let $Q=Q(a_0,a_1,\dots, a_d) \subseteq \R$ be a combinatorial cube. 
	Then there is an absolute constant $c>0$ such that 
\[
	\E^\times (Q) := |\{ (q_1,q_2,q_3,q_4) \in Q^4 ~:~ q_1 q_2 = q_3 q_4 \}| \ll |Q|^{3-c} \,.
\]
	Moreover, 
\[
 	|QQ|\gtrsim |Q|^{100/79},\,~~  |Q/Q| \gtrsim |Q|^{14/11} \quad  \quad \mbox{ and } \quad \quad |Q Q|, |Q/Q| \gg \min\{ |Q|^{6/5} , \sqrt{|Q| |\F|} \} 
\]
	for $\F = \R$ and $\F = \F_p$, 
	correspondingly. 
\label{t:Q_intr}
\end{theorem}

Here,  as always, we write $A+B$ for the sumset of sets $A$, $B$, further, $AB$ for the product set of $A,B$ and so on.
In other words,  
\[
	A+B:=\{a+b ~:~ a\in{A},\,b\in{B}\}\,, \quad AB:=\{ab ~:~ a\in{A},\,b\in{B}\} \,, 
\]
$$A/B:=\{a/b ~:~ a\in{A},\,b\in{B},\,b\neq0\}\,.$$

Finally, it is possible to  replace the addition to the multiplication in definition \eqref{def:cube_intr}, namely, one can consider  
\begin{equation}\label{def:cube_intr_mult}
	Q^\times (a_0,a_1,\dots, a_d) =  \left\{ a_0 \prod_{j=1}^d a^\eps_j ~:~ \eps_j \in \{0,1\} \right\} \,.
\end{equation}
Then we obtain an analogue of   Theorem \ref{t:Q_intr} for such cubes.

\begin{theorem}
	Let $Q=Q^\times (a_0,a_1,\dots, a_d) \subseteq \R$ be a combinatorial cube.
	Then there is an absolute constant $c>0$ such that 
\begin{equation}\label{f:Q_intr2_en}
	\E^+ (Q)  := |\{ (q_1,q_2,q_3,q_4) \in Q^4 ~:~ q_1 + q_2 = q_3 + q_4 \}| \ll |Q|^{3-c} \,.
\end{equation}
Moreover, 
\[
|QQ| \gtrsim |Q|^{100/79},\,~~ |Q/Q| \gg |Q|^{14/11} \quad  \quad \mbox{ and } \quad \quad |Q Q|, |Q/Q| \gg \min\{ |Q|^{31/30} , \sqrt{|Q| |\F|} \} 
\]
for $\F = \R$ and $\F = \F_p$, correspondingly.\\ 
	Finally, in $\F_p$ estimate \eqref{f:Q_intr2_en} takes place, provided  $|Q| \le p^{13/23}$. 
\label{t:Q_intr2}
\end{theorem}

The author is grateful to Jozsef Solymosi for useful discussions.

\section{Definitions and notation}

	Let $\Gr$ be an abelian group.
	Put
	$\E^{+}(A,B)$ for the {\it common additive energy} of two sets $A,B \subseteq \Gr$
	(see, e.g., \cite{TV}), that is, 
	$$
	\E^{+} (A,B) = |\{ (a_1,a_2,b_1,b_2) \in A\times A \times B \times B ~:~ a_1+b_1 = a_2+b_2 \}| \,.
	$$
	If $A=B$, then  we simply write $\E^{+} (A)$ instead of $\E^{+} (A,A)$
	and the quantity $\E^{+} (A)$ is called the {\it additive energy} in this case. 
	More generally, 
	we deal with 
	a higher energy
	\begin{equation}\label{def:T_k_ab}
	\T^{+}_k (A) := |\{ (a_1,\dots,a_k,a'_1,\dots,a'_k) \in A^{2k} ~:~ a_1 + \dots + a_k = a'_1 + \dots + a'_k \}| 
	\,.
	\end{equation}
	Another sort of a higher energy is 
	\[
	\E^{+}_k (A) = |\{ (a_1,\dots,a_k,a'_1,\dots,a'_k) \in A^{2k} ~:~ a_1-a'_1 = \dots = a_k - a'_k \}| \,.
	\]
	Sometimes we  use representation function notations like $r_{A+B} (x)$ or $r_{A+A-B}$, which counts the number of ways $x \in \Gr$ can be expressed as a sum $a+b$ or as a sum $a+a'-b$ with $a,a'\in A$, $b\in B$, respectively. 
	For example, $|A| = r_{A-A}(0)$ and  $\E^{+} (A) = r_{A+A-A-A}(0)=\sum_x r^2_{A+A} (x) = \sum_x r^2_{A-A} (x)$.  
	Having any functions $f_1,\dots, f_{k+1} : \Gr \to \C$ denote by
	$$ \Cf^{+}_{k+1} (f_1,\dots,f_{k+1}) (x_1,\dots, x_{k})$$
	the function
	$$
	    \Cf^{+}_{k+1} (f_1,\dots,f_{k+1}) (x_1,\dots, x_{k}) = \sum_z f_1 (z) f_2 (z+x_1) \dots f_{k+1} (z+x_{k}) \,.
	$$
	For example, $\Cf^{+}_2 (A,B) (x) = r_{B-A} (x)$.
	If $f_1=\dots=f_{k+1}=f$, then write
	$\Cf^{+}_{k+1} (f) (x_1,\dots, x_{k})$ for $\Cf^{+}_{k+1} (f,\dots,f_{}) (x_1,\dots, x_{k})$, where $f$ is taken $k+1$ times.

	If the group operation is the multiplication, then one can define 
	the {\it common additive energy} of two sets $A,B \subseteq \Gr$, namely, $\E^\times (A,B)$, the {\it multiplicative energy} $\E^{\times} (A)$ of $A$,  
	and so on.
	For example, we have   $\E^{\times} (A) = \sum_x r^2_{AA} (x)$.  
	In a similar way we  define $\Cf^{\times}_{k+1} (f_1,\dots,f_{k+1}) (x_1,\dots, x_{k})$ for arbitrary functions $f_1,\dots, f_{k+1} : \Gr \to \C$.

	\bigskip

	Now say a few words about combinatorial cubes.
	Let $h$ be a positive integer,  $a_0\in \Gr$ and $A = \{a_1,\dots,a_d\} \subseteq \Gr$ be a multi--set with $a_j \neq 0$, $j\in [d]$. 
	The {\it combinatorial cube} is the following set 
	\[
	Q_h=Q^A_h := a_0 + (Q^A_h)' = a_0 + \{0,a_1\} + \dots + \{0,a_d\} = \left\{ a_0 + \sum_{j=1}^d \eps_j a_j ~:~ \eps_j \in \{0,1,\dots,h \} \right\} \,.
	\]
	The number $d$ is called the {\it dimension} of $Q_h$ and $h$ is the {\it height} of $Q_h$.
	If $h=1$, then we write just $Q$ for $Q^A_1$. 
	Size of $Q_h$ can vary from $2$ (if all $a_j$ coincide and equal a non--zero element of order two) to $(h+1)^d$.
	In the last case $Q$ is called {\it proper}.   
	Having a set $X\subseteq [d]$ we put
	\[
	Q_h (X) := \left\{ a_0 + \sum_{j=1}^d \eps_j a_j ~:~ \eps_j \neq 0 \implies j\in X  \right\} \subseteq Q_h \,.
	\]
	Thus $Q_h  = Q_h ([d])$. 
	Clearly, if $X\sqcup Y = [d]$, then $Q_h=Q_h (X) +Q_h (Y)$. 
	In particular, $|Q_h | \le |Q_h (X)| |Q_h (Y)|$.
%
%
	Finally, put $U = h\sum_{j=1}^d a_j$. 
	Then $Q'_h = U-Q'_h$ and hence we have the following symmetric relation for any combinatorial cube 
	\begin{equation}\label{f:symm_cube}
	Q_h = (U+2a_0) - Q_h \,.
	\end{equation}

	More generally, having a finite set $\mathcal{D} \subseteq \Gr$, $|\mathcal{D}| \ge 2$, as well as some non--zero elements $a_0,a_1,\dots, a_d \in \Gr$ one can define $Q^A_{\mathcal{D}}$ (it can be associated with a {\it set with missing digits} see, e.g., \cite{Schmidt_digits}) 
\[
	Q_{\mathcal{D}} = Q_{\mathcal{D}}^A = \left\{ a_0 + \sum_{j=1}^d \eps_j a_j ~:~ \eps_j \in \mathcal{D}  \right\} \,.
\] 
	In other words, $Q^A_{\mathcal{D}}  = Q^A_h$ for $\mathcal{D} = \{0,1,\dots, h\}$. 
	Clearly,  $Q^A_{\mathcal{D}}$ does not enjoy property  \eqref{f:symm_cube} but again for $X\sqcup Y = [d]$ one has $Q_{\mathcal{D}}^A (X) + Q_{\mathcal{D}}^A (Y) = Q_{\mathcal{D}}^A ([d])$.

All logarithms are to base $2.$ The signs $\ll$ and $\gg$ are the usual Vinogradov symbols.
If we have a set $A$, then we will write $a \lesssim b$ or $b \gtrsim a$ if $a = O(b \cdot \log^c |A|)$, $c>0$.
When the constants in the signs  depend on a parameter $M$, we write $\ll_M$ and $\gg_M$. 
For a positive integer $n,$ 
let 
$[n]=\{1,\ldots,n\}.$ 
Throughout the paper by  $p$  we always  mean an odd prime number  and we put 
$\F_p = \Z/p\Z$.
If we consider a general field, then we write $\F$ to do not specify either $\F=\R$ or $\F=\F_p$.  

\section{Preliminaries}

Let $q$ be a prime power. 
Also, let $\mathcal{P} \subseteq \F_q^3$ be a set of points  and $\Pi$ be a collection of planes in $\F_q^3$. 
Having $r\in \mathcal{P}$ and $\pi \in \Pi$, we write 
\begin{displaymath}
\I (r,\pi) = \left\{ \begin{array}{ll}
1 & \textrm{if } r\in \pi\\
0 & \textrm{otherwise.}
\end{array} \right.
\end{displaymath}
	Denote by 
	$\mathcal{I}(\mathcal{P}, \Pi) = \sum_{r\in \mathcal{P}} \sum_{\pi \in \Pi} \I (r,\pi)$ the number of incidences between the  points $ \mathcal{P}$  and the planes $\Pi$ 
	and similarly the number $\mathcal{I}(\mathcal{P}, \mathcal{L})$ of incidences between  a collection of points $ \mathcal{P}$  and 
	a family 
	of lines $\mathcal{L}$. 
	The modern form of the points--lines, points--planes incidences for Cartesian products in $\F_p$, see \cite{SdZ}, \cite{Rudnev_pp}, as well as \cite{sh_as}.

\begin{theorem}
	Let $A,B \subseteq \F_p$ be sets, 
	$\mathcal{P} = A\times B$, and $ \mathcal{L}$ be a collection of lines in $\F^2_p$.
	Then 
	\begin{equation}\label{f:line/point_as}
	\mathcal{I}(\mathcal{P}, \mathcal{L}) - \frac{|A| |B| |\mathcal{L}|}{p} \ll |A|^{3/4} |B|^{1/2} |\mathcal{L}|^{3/4} + |\mathcal{L}| + |A| |B| \,.
	\end{equation}
\label{t:SdZ} 
\end{theorem}


\begin{theorem}
	Let $p$ be an odd prime, $\mathcal{P} \subseteq \F_p^3$ be a set of points and $\Pi$ be a collection of planes in $\F_p^3$. 
	Suppose that $|\mathcal{P}| \le |\Pi|$ and that $k$ is the maximum number of collinear points in $\mathcal{P}$. 
	Then the number of point--planes incidences satisfies 
	\begin{equation}\label{f:Misha+_a}
	\mathcal{I} (\mathcal{P}, \Pi)  - \frac{|\mathcal{P}| |\Pi|}{p} \ll |\mathcal{P}|^{1/2} |\Pi| + k |\Pi| \,.	
	\end{equation}
	\label{t:Misha+}	
\end{theorem}

We formulate  the best current result on the sum--product phenomenon in $\F_p$, see \cite[Theorem 1.2]{RSS} in a convenient way for us.

\begin{theorem}
	Let $A \subseteq \F_p$, $\la\neq 0$ and $|AA|=M|A|$, $|(A+\la)(A+\la)| = K|A|$. 
	If $|A| \le p^{36/67}$, then $\max\{K,M\} \gtrsim |A|^{2/9}$.
	The same is true if one replaces the multiplication to the division and vice versa.  
\label{t:RSS_2/9}	
\end{theorem}

Using growth in the affine group it was proved in \cite[Theorem 9, Lemma 21]{RSh} (the authors consider the case $A=B=C$ only but the arguments work in general case as well) that

\begin{theorem}
	Let $A,B,C \subseteq \R$ be finite sets, and $\kappa >0$ be any real.
	Suppose that $|C|^\kappa \le |B| \le |C|^2$. 
	Then there is $\d=\d(\kappa) > 0$ such that 
	\[
	\sum_{x} r^2_{B(A+C)} (x),\,~ \sum_{x} r^2_{BA+C} (x)  \ll |A|^{4/3} |B|^{3/2-\d} |C|^{5/3} \,. 
	\]
	\label{t:4.5-eps} 
\end{theorem}

Applying growth in the modular group we have obtained \cite[Theorem 1]{sh_Kloosterman}.

\begin{theorem}
	Let $A,B,C,D\subseteq \F_p$ be sets.
Then for any $\la \neq 0$, one has 
\[
|\{ (a,b,c,d) \in A\times B \times C \times D ~:~ (a+b) (c+d) = \lambda \}| - \frac{|A||B||C||D|}{p} 
\lesssim 
\]
\begin{equation*}\label{f:hyp_incidences_intr}
\lesssim |A|^{1/4} |B||C| |D|^{1/2} + |A|^{3/4} (|B||C|)^{41/48} |D|^{1/2}   \,.
\end{equation*}
\label{t:sh_Kloosterman} 
\end{theorem}

We finish this incidences part of section Preliminaries  by  the famous Szemer\'edi--Trotter Theorem \cite{sz-t}. 
Recall that  a set $\mathcal{L}$ of continuous plane curves a {\it pseudo--line system} if any two members of $\mathcal{L}$ have at
most one point in common. 

\begin{theorem}\label{t:SzT}
	Let $\mathcal{P}$ be a set of points and let $\mathcal{L}$ be a set of pseudo--lines in $\R^2$. 
	Then
	$$
	\mathcal{I}(\mathcal{P},\mathcal{L}) 
	\ll |\mathcal{P}|^{2/3}|\mathcal{L}|^{2/3}+|\mathcal{P}|+|\mathcal{L}|\,.$$
\end{theorem}

	The next result is essentially contained in \cite[Lemma 10]{s_sumsets} and also see the proof of \cite[Theorem 3]{s_sumsets}.

\begin{theorem}
	Let $A \subseteq \Gr$ be a set. 
	Suppose there are parameters $D_1$, $D_2$ such that $\E_3 (A) \le D_1 |A|^3$ and for any set $B \subseteq \Gr$ one has 
\[
	\E(A,B) \le D_2 |A| |B|^{3/2} \,.
\] 
	Then 
\[
	|A+A| \gtrsim |A|^{58/37} D^{-16/37}_1 D^{-10/37}_2 \,,
\]
	and 
\[
	|A-A| \gtrsim |A|^{8/5} (D_1 D_2)^{-2/5} \,.
\]
\label{t:s_sumsets}
\end{theorem}


Let us formulate a result from \cite[Lemma 4.1]{H_P} (it is formulated for $|\mathcal{D}|=2$ but the proof of the general case is the same). 

\begin{lemma}
	Let $Q_\mathcal{D} (A)$ be a cube. 
	One can split $[d]$ as a disjoint union of two sets $X$ and $Y$ such that $|Q_\mathcal{D} (X)| \le |Q_\mathcal{D} (Y)| \le |\mathcal{D}| |Q_\mathcal{D} (X)|$.   	
\label{l:H_P} 
\end{lemma} 


Finally, we need a combinatorial result \cite[Theorem 1.2]{GMR}.

\begin{theorem}
	Let  $k\ge 2$ and $A_1,\dots, A_k \subseteq \Gr$ be finite non--empty sets. 
	Put
\[
	S = A_1+\dots+A_k  
	\quad \quad \mbox{ and } \quad \quad S_j = A_1+\dots+A_{j-1}+A_{j+1} + \dots + A_k \,.
\] 
	Then 
\[
	|S|^{k-1} \le \prod_{j=1}^k |S_j| \,. 
\]
\label{t:GMR}
\end{theorem}

Theorem \ref{t:GMR} has 

\begin{corollary}
	Let $S_1,\dots, S_5 \subseteq \F_p$ be sets, $|S_j| \ge 2$ and $S=S_1+\dots+S_5$.
	Then 
\[
	|SS|, |S/S| \gg \min \{ |S|^{26/25}, |S|^{2/5} p^{1/2} \} \,.
\]
\label{c:Ruzsa_Sarkozy}
\end{corollary}
\begin{proof} 
	We consider the case $SS$ because for $S/S$ the argument is similar. 
	Let $\Pi = SS$. 
	Taking two different elements $\a,\beta \in S_5$, we have  $S_1+\dots+S_4 \subseteq (S - \a) \cap (S - \beta)$. 
	Put $x= \a-\beta \neq 0$ and let us estimate size of  $S\cap (S-x)$. 
Applying Theorem \ref{t:SdZ}, we have 
\[
|S_1+\dots+S_4| \le |S\cap (S-x)| 
	\le |S|^{-2} |\{ (\pi_1,\pi_2, q_1,q_2) \in \Pi^2 \times S^2 ~:~ \pi_1/q_1 - \pi_2/q_2 = x \}| 
\ll 
\]
\[
\ll 
\frac{|SS|^2}{p} + |SS|^{5/4} |S|^{-1/2} + |SS||S|^{-1} 
\ll 
\frac{|SS|^2}{p} + |SS|^{5/4} |S|^{-1/2} \,.
\]
The same holds for all $j\in [5]$. 
Using Theorem \ref{t:GMR}, we obtain 
\[
|S|^4 \ll \left( \frac{|SS|^2}{p} + |SS|^{5/4} |S|^{-1/2} \right)^5 \,.
\]
It gives us 
\[
|SS| \gg \min \{ |S|^{26/25}, |S|^{2/5} p^{1/2} \} 
\]
as required. Notice that a similar argument was used in \cite{Sarkozy_R}. 
	This completes the proof. 
$\hfill\Box$
\end{proof}

\section{Proper cubes}

In this section we consider 
proper cubes.  
The results here are auxiliary but they show transparently  that such cubes have strong additive properties 
(and hence we can hope to demonstrate that combinatorial cubes have rather weak multiplicative behaviour).
To do this 
we calculate some additive characteristics  of proper cubes.

Let $l\ge 1$ be an integer.  
Take  a vector $\vec{x} = (x_1,\dots,x_d)$ with $x_j \le l$.
For any such vector we write $n_{j} = |\{ i\in [d] ~:~ x_i = j \}|$, $0 \le j \le l$. 
Clearly, $\sum_{j=0}^{h} n_j = d$ and  
we say that $\vec{x}$ has {\it type} $(n_{0},\dots,n_l)$. 
We write $p_{k,h} (m)$ for the number of solutions to the equation 	$c_1+\dots+c_k = m$, $0\le c_i \le h$. 
Clearly, $p_{k,1} (m) = \binom{k}{m}$.  
For the  general theory of partitions consult, e.g., \cite{Andrews}.

\begin{lemma}
	Let $h,k,l$, $l \le kh$ be positive integers.
	Also, let a vector $\vec{x}$ has type $(n_{0},\dots,n_l)$ and let $Q_h$ be a proper combinatorial cube. 
	Then $|kQ_h| \le |Q_h|^{\log_{h+1} (kh+1)}$ and 
	\begin{equation}\label{f:Q^(k)_T_k}
		r_{kQ_h}
		(\vec{x}) \ge \prod_{j=1}^{kh} (p_{k,h} (j))^{n_j}  \,.
	\end{equation}
	In particular, 
\[
	\T^{+}_k (Q_h) \gg 
		|Q_h|^{2k-1 - O(\log_{h+1} k)} 
	\,, \quad \quad  
	\T^{+}_k (Q) \ge  |Q|^{2k - 1 - \frac{\log k}{2}}  \,, 
\]
	and 
\[
	\E^{+}(Q_h) \ge |Q_h|^{k+\frac{h^{k+1}}{(k+1) (h+1)^k \ln (h+1)}} \,, \quad \quad  
		\E^{+}_k (Q) \ge |Q|^{k + 2^{-k}} \,.
\]
\label{l:Q^(k)_T_k}
\end{lemma}
\begin{proof}
	Put $H=\{0,1,\dots,h\}$.  
	The bound $|kQ_h| \le |Q_h|^{\log_{h+1} (kh+1)}$  follows from the fact that 
	$$
		kQ_h \subseteq ka_0 + \sum_{j=1}^d \{0,1,\dots,kh\} \cdot a_j =  ka_0 + \sum_{j=1}^d kH \cdot a_j \,.
	$$
	Further take any $j$ such that  $0\le j\le l$ and consider positions $S \subseteq [d]$ of $\vec{x}$ with $x_i =j$. 
	Then the number of representations of  any $s\in S$ as a sum of $k$ elements from $Q_h$ equals the number of the solutions to the equation 
	$c_1+\dots+c_k = j$, $0\le c_i \le h$.
	In other words, this is  $(p_{k,h} (j))^{n_j}$ and  hence we obtain \eqref{f:Q^(k)_T_k}. 
	To calculate $\T^{+}_k (Q_h)$ we sum the previous bound (or use the direct argument) and crudely estimating the sum  $\T^{+}_k (H)$ via dispersion,  
	to get 
\[
	\T^{+}_k (Q_h) \ge \sum_{n_j} \frac{d!}{\prod_j n_j!} \prod_j p_{k,h} (j)^{2n_j} 
	=
	\left(\sum_{j=0}^{kh} p^2_{k,h} (j) \right)^d =  (\T^{+}_k (H))^d 
	\gg
\]
\[ 
	\gg
	\frac{(h+1)^{2kd}}{O((\sqrt{k}h)^d)}
	\gg 
	|Q_h|^{2k-1 - O(\log_{h+1} k)} \,. 
\]
	To obtain the required lower bound for $\T^{+}_k (Q)$ we use the formula  $p_{k,1} (m) = \binom{k}{m}$ and make the previous calculation to get
\[
	\T^{+}_k (Q) \ge \left(\sum_{j=0}^{k} \binom{k}{j}^2 \right)^d = \binom{2k}{k}^d 
		\ge
	\left( \frac{2^{2k}}{2\sqrt{k}} \right)^d \ge |Q|^{2k - 1 - \frac{\log k}{2}} \,.	
\]

	Similarly, because the number of the solutions to the equation $c_1-c_2 =j$ is $h-|j|+1$, where  $0\le c_1,c_2 \le h$ and $j$ is any number with $|j|\le h$, we have 
\[
	\E^{+}_k (Q_h) \ge \sum_{n_j} \frac{d!}{\prod_j n_j!} \prod_j (h-|j|+1)^{n_j k} 
	= 
	\left( \sum_{|j|\le h} (h-|j|+1)^k \right)^d
	=
	((h+1)^k + 2 \sum_{m=1}^h m^k )^d
	\ge
\]
\[
	\ge
	\left((h+1)^k + \frac{2 h^{k+1}}{k+1} \right)^d
	\ge
	|Q_h|^{k+\frac{h^{k+1}}{(k+1)(h+1)^k \ln (h+1)}} \,,
\]
	and for $h=1$, we get
\[
	\E^{+}_k (Q) \ge (2^k+2)^d \ge |Q|^{k + 2^{-k}} \,.
\]
	This completes the proof. 
	$\hfill\Box$
\end{proof}

\bigskip

Now we show that proper combinatorial cubes cannot be closed under the multiplication in a rather strong sense.

\begin{theorem}
	Let $\F$ be either $\R$ or $\F_p$ and $h\ge 1$ be a positive integer.
	If $Q_h$ is a proper cube, $|Q_h| < |\F|^{24/49}$, then there is an absolute constant $c>0$ such that 
	\begin{equation}\label{f:Q_energy}
		\E^\times (Q_h) \ll |Q_h|^{3-c}
	\end{equation} 
	further in $\R$ 
\begin{equation}\label{f:Q_energy_R}
	\E^\times (Q_h) \ll |Q_h|^{\frac32 + \log_{h+1} (2h+1)}\,,
\end{equation}
	and in $\F_p$ for any proper cube $Q_h$ 
	\begin{equation}\label{f:Q_energy+}
	\E^\times (Q_h) \lesssim  \frac{|Q_h|^{3+\log_{h+1} (2h+1)}}{p} +  \min\{ |Q_h|^{2 + \frac23 \log_{h+1} (2h+1) }, |Q_h|^{1 + \frac32 \log_{h+1} (2h+1) } \} \,.
	\end{equation}
\label{t:Q_mult_energy}
\end{theorem}
\begin{proof}
	Let $Q=Q_h$. 
	Take 
	 a parameter $\tau\le |Q|$ and consider the set $\Omega_\tau$ such that $r_{QQ} (\omega) \ge \tau$ for any $\omega \in \Omega_\tau$.
	Using the Szemer\'edi--Trotter Theorem \ref{t:SzT}, we have
	\[
		\tau |Q| |\Omega_\tau| \le |\{ (q_1,q_2,s,\o)\in Q^2 \times (Q+Q) \times \Omega_\tau ~:~ q_1 (s-q_2) = \o \}| 
		\ll
	\]
	\[
		\ll  
		|\Omega_\tau|^{2/3} |Q|^{4/3} |Q+Q|^{2/3} + |\Omega_\tau| |Q| + |Q+Q| |Q| \,,
	\]
	and hence for $\tau \gg 1$, we obtain  
	\[
		|\Omega_\tau| \ll \max\{|Q+Q|^2 |Q| \tau^{-3}, |Q+Q| \tau^{-1} \} \ll |Q+Q|^2 |Q| \tau^{-3} \,.
	\]
	Thus after the summation over $\tau$, we see that 
	\[
		\E^{\times} (Q) \ll |Q+Q| |Q|^{3/2} \ll |Q|^{3/2 + \log_{h+1} (2h+1)} \,.
	\]
	If $\F = \F_p$, then we use Theorem \ref{t:SdZ} to derive
	\[
		\tau |Q| |\Omega_\tau| \ll \frac{|Q|^2|Q+Q||\Omega_\tau|}{p} + |\Omega_\tau|^{3/4} |Q+Q|^{1/2} |Q|^{3/2} + |\Omega_\tau| |Q| + |Q+Q| |Q|
	\]
	whence 
	\[
		\E^\times (Q) \ll \frac{|Q|^3|Q+Q|}{p} + |Q+Q|^{2/3} |Q|^2 \,.
	\]
	Let us obtain another bound.
	Using Theorem \ref{t:SdZ} again, we get
	\[
		\tau |Q| |\Omega_\tau| \ll \frac{|Q|^2|Q+Q||\Omega_\tau|}{p} + |\Omega_\tau|^{1/2} |Q+Q|^{3/4} |Q|^{3/2} + |\Omega_\tau| |Q| + |Q+Q| |Q|
	\]
	It gives us
	\[
		\E^\times (Q) \lesssim  \frac{|Q|^3|Q+Q|}{p} + |Q+Q|^{3/2} |Q| 
	\]
	as required.

	To obtain \eqref{f:Q_energy}, we write $\E^{\times} (Q) = |Q|^3/M$, where $M\ge 1$ is a number and we need to obtain a good lower  bound for $M$. 
	Using the Balog--Szemer\'edi--Gowers Theorem (see, e.g., \cite{TV}), we find $B \subseteq Q$ such that $|B| \gg_M |Q|$, $|BB| \ll_M |B|$. 
	We have
\[
	|B\pm B| \le |Q\pm Q| \le (2h+1)^d \le  |Q|^{3/2} \,.
\]
	But by the main result of \cite{s_weak}, we also have $|B-B| \gg_M |B|^{5/3}$ in $\R$ (there is  an analogues result for  $B+B$) and $|B-B| \gg_M |B|^{3/2+1/24}$ for $B$ in $\F_p$, $|B| <p^{24/49}$  see \cite[Theorem 4]{RSS}. 
	This completes the proof. 
	$\hfill\Box$
\end{proof}

\begin{remark}
	\label{r:proper_better}
	Applying the arguments from the proof of \cite[Lemma 4]{RS} one can improve  the dependence on $h$ in \eqref{f:Q_energy_R}, \eqref{f:Q_energy+}  for large $h$.
	We do not make such calculations.   
\end{remark}

\begin{remark}
\label{r:SZ}	
By  \cite[Corollary 1, Remarks 1,3]{SZ} we have  for any $B,C \subset \R$ with, say,  $|B| \sim |C| \ge 2$ that 
\begin{equation}\label{f:H_P}
	\E^\times (B+C) \ll |B|^{6-c} \,,
\end{equation}
where $c>0$ is an absolute constant. 
Obviously, we can split any proper cube $Q_h$ as  $Q_h = Q_h (X) + Q_h (Y)$, $|Q_h| =  |Q_h (X)| |Q_h (Y)|$ and $|Q_h (X)| \sim |Q_h (Y)|$   
(more generally, by Lemma \ref{l:H_P} for any cube one can split $[d]$ as a disjoint union of some sets $X$ and $Y$ such that $|Q_h (X)| \le |Q_h (Y)| \le h|Q_h (X)|$).   
Applying \eqref{f:H_P} with $B=Q_h (X)$, $C=Q_h (Y)$, we obtain a non--trivial upper bound for the multiplicative energy of any cube $Q_h$, provided 
$|Q_h (X)| |Q_h (Y)| \ll_h |Q_h|^{1+o(1)}$. 
In particular, this condition takes place if $Q_h$ is a proper cube. 
It gives an alternative proof of estimate \eqref{f:Q_energy}. 
\end{remark}

\section{General cubes and sets with missing digits}

Now we consider the case of general cubes \eqref{def:cube_intr}. 
It is relatively easy to see that such cubes must grow under multiplication, e.g., because they contains different shifts of subcubes of smaller dimension.
Nevertheless, the obtaining of 
upper bounds on different types of energies of the cubes is a  more delicate question. 
The main difference between  our new results and paper \cite{H_mult} is that they do not depend on the sumsets/the product sets of the considered cubes.

We start with a simple but a crucial combinatorial lemma.

\begin{lemma}
	Let $Q \subseteq \Gr$ be a combinatorial cube, $B\subseteq Q$ be  any set and let $h=1$. 
	Then there are two sets $S \subseteq Q+Q$, $D\subseteq Q-Q$ with $|S|, |D| \le |Q|^{3/2}$ such that for any $b_1,b_2 \in B$ either $b_1-b_2 \in D$ or $b_1+b_2 \in S$. 
	In particular, 
	\begin{equation}\label{f:Q^(k)}
	\E^{+} (B,Q) \ge |B|^2 |Q|^{1/2} \,.
	\end{equation}
	\label{l:Q^(k)}
\end{lemma}
\begin{proof}
	Write $Q$ for $Q_h$.
	In the proof we use some parts of the arguments of the proof of Lemma \ref{l:Q^(k)_T_k}.  
	Let $x=b_1+b_2 \in B+B$, where $b_1,b_2 \in Q$.   
	For an arbitrary integer $h$ any $x\in Q+Q$ can be written as $x=\sum_{j\in P_1} \eps_j a_j + \sum_{j\in P_2} \tilde{\eps}_j a_j$, where $1\le \eps_j \le h$ on $P_1 \subseteq [d]$, and  $h<\tilde{\eps}_j \le 2h$ on   $P_2 \subseteq [d]$.
	Clearly, $P_1,P_2$  are disjoint sets and we put $Z = [d]\setminus (P_1 \sqcup P_2)$. 
	Now let us use that $h=1$ in our case. 
	Write $x= \sum_{j\in P_2} a_j + \sum_{j\in P_2} a_j + y_1 + y_2$, where $y_1,y_2 \in Q (P_1)$ such that $y_1+y_2 = \sum_{j\in P_1} \eps_j a_j$.
	Clearly, we have at least $|Q(P_1)|$ ways of writing $x$ this way and hence $r_{Q+Q} (x) \ge |Q (P_1)|$. 
	Further consider $x^*=b_1-b_2 = \sum_{j\in P'_1} \eps'_j a_j - \sum_{j\in P''_1} \eps''_j a_j$, where $1\le \eps'_j, \eps''_j \le h$ and $P'_1 \sqcup P'_2 = P_1$. 
	As above 
	we have at least $|Q (P_2 \sqcup Z)|$ ways of writing $x^*$ as $x^*=q^*_1-q^*_2$, $q^*_1, q^*_2 \in Q$. 
	Thus 
	\[
	r_{Q+Q} (q_1+q_2) + r_{Q-Q} (q_1-q_2) \ge |Q (P_1)| + |Q (P_2\sqcup Z)| \ge 2 \sqrt{|Q (P_1)||Q (P_2\sqcup Z)|} \ge 2|Q|^{1/2} \,.
	\] 
	Summing the last bound  over $q_1,q_2\in B$ and using the H\"older inequality, we obtain 
	\[
	2\E^{+}(B,Q)  \ge 2 |Q|^{1/2} |B|^2  
	\]
	as required. 
	Finally, either $|Q (P_1)|$ or $|Q (P_2\sqcup Z)|$ is at least $\sqrt{|Q|}$. 
	Hence either $r_{Q-Q} (b_1-b_2) \ge \sqrt{|Q|}$ or $r_{Q+Q} (b_1+b_2) \ge \sqrt{|Q|}$.
	It remains to define 
\begin{equation}\label{def:S,D} 
	S = \{ x \in Q+Q ~:~ r_{Q+Q} (x) \ge \sqrt{|Q|} \},\,   \quad \quad D = \{ x\in Q-Q ~:~ r_{Q-Q} (x) \ge \sqrt{|Q|} \} 
\end{equation}
	 and notice that $|S|,|D| \le |Q|^{3/2}$.  
	This completes the proof. 
	$\hfill\Box$
\end{proof}

\bigskip

\begin{remark}
	Estimate \eqref{f:Q^(k)} is the best possible up to factors $|Q_h|^{o(1)}$. 
	Indeed, just consider a proper cube $Q_h$ such that $Q_h+Q_h$ is also a proper cube.
	Further take  $B \subset Q_h$ such that for any $b\in B$, $b=a_0 + h\sum_{j \in S_b} a_j$, where the set $S_b$ is taken randomly with probability $1/2$. Then with high probability for any $b,b' \in B$ we have $r_{Q_h+Q_h} (b+b'), r_{Q_h-Q_h} (b-b') \ll |Q_h|^{1/2+o(1)}$ and hence 
	$\E^{+} (B,Q_h) \ll |B|^2 |Q_h|^{1/2+o(1)}$.     
	For $h=1$ the set $B$ is large, namely, $|B| \gg |Q|^{1-o(1)}$ but for $h>1$ is not. 
\end{remark}


Now we obtain the first main result on growth of combinatorial cubes. 
The bounds below can depend on the height $h$. 
The constant $c$ in \eqref{f:Q_prod_D} can be taken $c=1/25$ in both fields freely.

\begin{theorem}
	Let $Q_h \subseteq \F$ be a combinatorial cube.
	Then in $\R$ 
	\begin{equation}\label{f:Q_prod}
	|Q_h Q_h| \gtrsim |Q_h|^{100/79}, \quad \quad  |Q_h/Q_h| \gtrsim |Q_h|^{14/11} 
	\end{equation}
	and in $\F_p$ 
	\begin{equation}\label{f:Q_prod'}
		|Q_h Q_h|, |Q_h/Q_h| \gtrsim \max\{ \min\{ |Q_h|^{6/5} , \sqrt{|Q_h| p} \},  |Q_h|^{11/9} \} \,,
	\end{equation}
	where the second bound in the maximum is applicable for $|Q_h| \le p^{36/67}$ only. 
	Both in $\R$ and in $\F_p$ one has for any finite set $\mathcal{D}$ and a certain $c>0$ that 
	\begin{equation}\label{f:Q_prod_D}
		|Q_{\mathcal{D}} Q_{\mathcal{D}}|,\, 
		|Q_{\mathcal{D}}/Q_{\mathcal{D}}| 
			\gg \min \{ |Q|^{1+c}, |Q|^{2/5} p^{1/2} \} \,.
	\end{equation}
	Further, for $h=1$ and $\F = \R$ 
	there is an absolute constant $c_*>0$ such that 
	\begin{equation}\label{f:energy_Q1}
	\E^\times (Q) \ll |Q|^{3-c_*} \,.
	\end{equation}
	\label{t:energy_Q1} 
\end{theorem}
\begin{proof}
	Let $Q=Q_h$. 
	First of all, we obtain a weaker result than \eqref{f:Q_prod} for $QQ$ and $Q/Q$.
	We restrict ourself considering the case $QQ$ only  because for $Q/Q$ the arguments are the same. 
	By \eqref{f:symm_cube} we see that the equation $x=(U+2a_0) - y$, $x,y\in Q$ has $|Q|$ solutions.
	In principle,  the number $U':=U+2a_0$ can be zero but then one can consider $Q_* = Q_{[d-1]}$, $|Q_*| \ge |Q|/h$ instead of $Q$. 
	Denote by $\sigma$ the number of the solutions to the equation $x=U'-y$, $x,y\in Q_*$ and below we use the same letter $Q$ for $Q_*$.
	One has 
	\begin{equation}\label{tmp:pi_1}
	\sigma \le |Q|^{-2} |\{ \pi_1/q_1 = U' - \pi_2/q_2  ~:~ q_1,q_2 \in Q,\, \pi_1,\pi_2 \in QQ\}| \,.
	\end{equation}
	Using the Szemer\'edi--Trotter Theorem in $\R$, we get 
	\[
	|Q|\ll \sigma \ll |Q|^{-2} (|QQ|^{4/3} |Q|^{4/3} + |QQ||Q|)
	\]
	and hence $|QQ| \gg |Q|^{5/4}$.
	To obtain improved  bound \eqref{f:Q_prod} just use Theorem \ref{t:s_sumsets} and notice that the parameters $D_1$, $D_2$ can be taken $D_1 = D^{2}_2 = (|QQ|/|Q|)^2$, see \cite{s_sumsets}. 
	The arguments similar to the arguments from  Remark \ref{r:SZ} give us \eqref{f:Q_prod_D} in $\R$. 
	Indeed, 
	we can split $Q = Q_{\mathcal{D}}$ as $Q=Q'+Q''$, where $2\le |Q'| \le |Q''|$ and apply the main result of \cite{SZ}, which says that for a certain $c>0$ the following holds 
	$$
		|QQ| = |(Q'+Q'') (Q'+Q'')| \gg |Q'+Q''|^{1+c} = |Q|^{1+c}
	$$
	and similar for $|Q/Q|$.

	In $\F_p$ we do the same, applying Theorem \ref{t:SdZ}. 
	Namely,
	\begin{equation}\label{tmp:pi_2}
	|Q|\ll \sigma \ll |Q|^{-2} \left( \frac{|QQ|^2 |Q|^2 }{p} + |QQ|^{5/4} |Q|^{3/2} + |QQ||Q| \right)	
	\end{equation}
	and thus $|QQ| \gg \min\{ |Q|^{6/5} , \sqrt{|Q|p} \}$. 
	If $|Q| \le p^{36/67}$, then we can apply Theorem \ref{t:RSS_2/9}. 
	To obtain \eqref{f:Q_prod_D} in the case of the finite field split $[d]$ onto five sets $X_1,\dots,X_5$ such that for $Q_j := Q_\mathcal{D} (X_j)$ 
	we have $|Q_j| \ge 2$. 
	It is possible to do because all $a_j$ do not vanish. 
	We have $Q = Q_1 + \dots +Q_5$. 
	Using Corollary \ref{c:Ruzsa_Sarkozy} we obtain the result.

	It remains to prove \eqref{f:energy_Q1}. 
	We write $\E^{\times} (Q) = |Q|^3/M$, where $M\ge 1$ is a number and we need to obtain a good lower  bound for $M$. 
	Using the Balog--Szemer\'edi--Gowers Theorem (see, e.g., \cite{TV}), we find $B \subseteq Q$ such that $|B| \gg_M |Q|$, $|BB| \ll_M |B|$. 
	Put $\Pi = BB$. 
	By Lemma \ref{f:Q^(k)} we have 
	\begin{eqnarray}\label{f:E_lower}
	\E^{+} (B,Q) \ge |B|^2 |Q|^{1/2} \,.
	\end{eqnarray} 
	Applying  Theorem \ref{t:4.5-eps}, we get  for a certain $\d>0$ that 
	\[
	\E^{+} (B,Q) \le |B|^{-2} |\{ (b,b',q,q',\pi,\pi') \in B^2\times Q^2 \times \Pi ~:~ \pi b + q = \pi' b + q \}| 
	\ll
	\]
	\begin{equation}\label{f:E_tmp} 
	\ll
		|B|^{-2} |\Pi|^{3/2-\d} |B|^{4/3} |Q|^{5/3}
	\end{equation}
	Since $|B| \gg_M |Q|$, $|\Pi| \ll_M |B|$ we see that 
	\[
		|Q|^{5/2} \ll_M |B|^2 |Q|^{1/2} \ll_M |Q|^{5/2-\d}
	\]
	In other words, $M\gg |Q|^{c_1}$, where $c_1>0$ is an absolute constant.  
	This completes the proof. 
	$\hfill\Box$
\end{proof}

\bigskip

Now we obtain an analogue of Theorem \ref{t:energy_Q1}  for multiplicative  combinatorial cube, see definition \eqref{def:cube_intr_mult}.

\begin{theorem}
	Let $Q^\times_h \subseteq \F$ be a multiplicative combinatorial cube.
	Then in $\R$ 
	\begin{equation}\label{f:Q_prod_m}
	|Q^\times_h + Q^\times_h| \gtrsim |Q^\times_h|^{100/79},\, \quad \quad   |Q^\times_h - Q^\times_h| \gtrsim  |Q^\times_h|^{14/11} \,, 
	\end{equation}
	and in $\F_p$ 
	\begin{equation}\label{f:Q_prod_m'}
		|Q^\times_h + Q^\times_h|, |Q^\times_h-Q^\times_h| \gg \min\{ |Q^\times_h|^{31/30} , \sqrt{|Q^\times_h| p} \} \,.
	\end{equation}
	Further, for $h=1$ 
	and $\F = \R$ 
	there is an absolute constant $c>0$ such that 
	\begin{equation}\label{f:energy_Q1_m}
	\E^+ (Q^\times) \ll |Q^\times|^{3-c} \,.
	\end{equation}
	\label{t:energy_Q1_m} 
\end{theorem}
\begin{proof}
	Let $Q=Q^\times_h$ and we consider the case of the addition only because for the subtraction the argument is the same. 
	The arguments which give \eqref{f:Q_prod} are applicable for  \eqref{f:Q_prod_m} if one replaces the addition to the multiplication because now we arrive to the equation of the hyperbolas $xy = \lambda$, which form a pseudo--line system. 
	The same concerns \eqref{f:energy_Q1_m} in the case $\F = \R$ because Theorem \ref{t:4.5-eps} works perfectly for the addition and for the multiplication.
	As for \eqref{f:Q_prod_m'}  we follow the same scheme but apply Theorem 
	\ref{t:sh_Kloosterman},  
	which gives us for any $\lambda \neq 0$ that 
\[
	|Q| \ll |Q|^{-2} |\{ (s,s',x,x')\in (Q+Q)^2 \times Q^2 ~:~ (s-x)(s'-x') = \lambda \}| 
	\lesssim 
\]
\[
	\lesssim 
	\frac{|Q+Q|^2}{p} + |Q+Q|^{3/4} + |Q+Q|^{5/4} |Q|^{-7/24} \,.
\]
	The last estimate implies \eqref{f:Q_prod_m'}. 
This completes the proof. 
$\hfill\Box$
\end{proof} 


\begin{remark}
	\label{r:gen_better}
	Again similar to Remark \ref{r:proper_better} one can apply the arguments from \cite{Olmezov_convex}, \cite{RS} to improve the constants in \eqref{f:Q_prod}, \eqref{f:Q_prod_m}. 
	It is possible to check that the constant $100/79$ can be replaced to $52/41$.  
	We leave these calculations for the interested reader.
	Instead of we use general Theorem \ref{t:s_sumsets} (which equally works in the case of the prime field) because our main aim is to obtain energy bounds.  
\end{remark}

Now we obtain a non--trivial bound for 
the additive energy 
of combinatorial cubes from $\F_p$, which defined in \eqref{def:cube_intr_mult}.
Probably, Theorem \ref{t:sigma_new}  below 
is the deepest result of our paper.
We need a combinatorial result, which is a small generalization of Lemma 2 from beautiful paper \cite{Olmezov_elementary} devoted to an elementarisation of the eigenvalues method see, e.g., \cite{s_sumsets}.

\begin{lemma}
	Let $A,B,D \subseteq \Gr$ be sets and  $1\le s< n$, $m\ge 1$ be positive integers.
	Then 
	\[
	\left( \sum_{x \in A} B(y) D(y-x)  \right)^{mn} \le |A|^{(n-1)m} |B|^{s(m-1)} |D|^{(n-s)(m-1)}  \times 
	\]
	\begin{equation}
	\times 
	\sum_{\vec{x}}  \sum_{\vec{y}} \Cf^{n-s}_m (B) (\vec{x}) 
	\Cf_{m+s} (A, A,\dots, A, B, \dots, B) (\vec{x}, \vec{y}) 
	\prod_{i=1}^s \prod_{j=1}^{m} D(y_i - x_j) \,, 
	\end{equation}
	where $\vec{x} = (x_2,\dots, x_m)$, $x_1=0$ and $\vec{y} = (y_1,\dots, y_s)$. 
\label{l:Olmezov}
\end{lemma}
\begin{proof} 
	For $x\in \Gr$ write $\D(x) = (x,x,\dots,x)$. 
	Also, let $S^k$, $k\ge 1$  denotes 
	the 
	Cartesian product of a 	set $S \subseteq \Gr$. 
	Using the H\"older inequality, we obtain 
	\[
	\sigma^n = \left( \sum_{x \in A} B(y) D(y-x)  \right)^n \le |A|^{n-1} \sum_{x \in A} \sum_{\vec{y}} B^{n} (\vec{y}) D^n (\vec{y}-\D(x))  
	=
	\]
	\begin{equation}\label{f:explanation_Olmezov}
	=
	|A|^{n-1} \sum_{\vec{y}_1, \vec{y}_2} B^{s} (\vec{y}_1) D^{n-s} (\vec{y}_2) \sum_{x \in A} B^{n-s} (\vec{y}_2+\D(x)) D^{s} (\vec{y}_1-\D(x)) \,.
	\end{equation}
	Here $\vec{y} = (\vec{y}_1, \vec{y}_2)$, $0\le s\le n$, the vector $\vec{y}_1$ has $s$ components and the vector  $\vec{y}_2$ has $(n-s)$ components. 
	Formula \eqref{f:explanation_Olmezov} shows the main idea of the proof: we can switch freely  the restrictions  on components of all obtained vectors between the inner and the external sums. 
	Now again  by the H\"older inequality, we derive
	\[
	\sigma^{nm} \le |A|^{(n-1)m} |B|^{s(m-1)} |D|^{(n-s)(m-1)} \times
	\]
	\[
	\times 
	\sum_{\vec{y}_1, \vec{y}_2} B^{s} (\vec{y}_1) D^{n-s} (\vec{y}_2) \sum_{x_1, \dots, x_m \in A}\, \prod_{j=1}^m B^{n-s} (\vec{y}_2+\D(x_j)) D^{s} (\vec{y}_1-\D(x_j))
	\le 
	\]
	\[
	\le |A|^{(n-1)m} |B|^{s(m-1)} |D|^{(n-s)(m-1)} 
	\sum_{\vec{y}_1, \vec{y}_2} B^{s} (\vec{y}_1) \sum_{x_1, \dots, x_m \in A}\, \prod_{j=1}^m  B^{n-s} (\vec{y}_2+\D(x_j)) D^{s} (\vec{y}_1-\D(x_j))
	\,.
	\]
	Let $\sigma^{nm}_1 =  \sigma^{nm} / |A|^{(n-1)m} |B|^{s(m-1)} |D|^{(n-s)(m-1)}$.  
	Summing over $\vec{y}_2$ and changing the variables, we get 
	\[
	\sigma^{nm}_1 \le 
	\sum_{\vec{y}_1} B^{s} (\vec{y}_1) \sum_{x_1, \dots, x_m \in A}\, \Cf^{n-s}_m (B) (x_2-x_1, \dots, x_m-x_{1})  
	\prod_{j=1}^m D^{s} (\vec{y}_1-\D(x_j))
	=
	\]
	\[
	=
	\sum_{\vec{y}_1} B^{s} (\vec{y}_1) A(x_1) A(x_1+x_2) \dots A(x_1+x_m) \Cf^{n-s}_m (B) (x_2, \dots, x_{m})  
	D^{s} (\vec{y}_1-\D(x_1)) \prod_{j=2}^m D^{s} (\vec{y}_1-\D(x_j+x_1))
	\]
	\[
	= 
	\sum_{\vec{y}_1} \sum_{x_2,\dots, x_m} \Cf^{n-s}_m (B) (x_2, \dots, x_{m}) \Cf_{m+s} (A, A,\dots, A, B, \dots, B) (x_2, \dots, x_{m}, \vec{y}_1) 
	D^s (\vec{y}_1) \prod_{j=2}^m D^{s} (\vec{y}_1-\D(x_j))
	\]
	as required. 
	$\hfill\Box$
\end{proof}

\begin{theorem}
	Let $Q^\times \subseteq \F_p$ be a combinatorial cube,  $|Q^\times| \le p^{13/23}$.  
	Then 	there is an absolute constant $c>0$ such that 
\begin{equation}\label{f:sigma_new}	
	\E^+ (Q^\times) \ll |Q^\times|^{3-c} \,.	
\end{equation}
\label{t:sigma_new}
\end{theorem}
\begin{proof} 
	Let $Q=Q^\times$. 
	As in the proof of inequality \eqref{f:energy_Q1} of Theorem \ref{t:energy_Q1} we write $\E^{+} (Q) = |Q|^3/M$, where $M\ge 1$ is a number and we need to obtain a good lower  bound for $M$. 
	Using the Balog--Szemer\'edi--Gowers Theorem (see, e.g., \cite{TV}), we find $B \subseteq Q$ such that $|B| \gg_M |Q|$, $|B+B| \ll_M |B|$. 
	By Lemma \ref{l:Q^(k)} we find the sets $D,S$ such that $|S|, |D| \le |Q|^{3/2}$ and either $\sum_{x\in S} r_{BB} (x) \ge |B|^2/2$ or $\sum_{x\in D} r_{B/B} (x) \ge |B|^2/2$.   
	Without loosing of the generality consider the first case. 
	Applying Lemma \ref{l:Olmezov} with the parameters $m=n=2$, $s=1$  to the sets $A=B^{-1}$, $B=B$, $D=S$, we obtain 
\[
	|B|^{8} \ll \left(\sum_{x\in S} r_{BB} (x) \right)^4 \le |B|^3 |S| \sum_{x,y} r_{B/B} (x) \Cf^{\times}_3 (B^{-1},B^{-1},B) (x,y) S(y) S(y/x) \,. 
\]  
	Using the H\"older inequality, we get 
\begin{equation}\label{tmp:25.07_1}
	|B|^{10} \ll |S|^2 \E^{\times}_3 (B) \sum_z r^2_{B/B} (z) r_{S/S} (z) \le |S|^3 \E^{\times}_3 (B) \E^{\times} (B) \,.
\end{equation}
	To estimate $\E^{\times}_3 (B)$,  $\E^{\times} (B)$   we apply \cite[Lemmas 23, 25]{collinear}.
	In terms of Theorem \ref{t:s_sumsets} these lemmas give us $D_1 = (|B\pm B|/|B|)^{15/4}$ and $D_2 = (|B\pm B|/|B|)^{3/2}$, provided 
	$|B|^{11} |B\pm B| \le p^{8}$ and $|B|^2 |B\pm B| \le p^2$. 
	Also, \cite[Theorem 35]{collinear} implies $\E^{\times} (B) \lesssim_M |B|^{32/13}$, provided $|B| \le p^{13/23}$.
	The restrictions to size of $B$ and $B\pm B$ 
	can be simplified as $|Q| \le p^{13/23}$ because one can assume that the parameter $M$ is sufficiently small.     
	Substituting the last bounds into \eqref{tmp:25.07_1} and recalling that $|B| \gg_M |Q|$, $|S| \le |Q|^{3/2}$, we 
	obtain 
\[
	|B|^{10} \lesssim_M  |S|^3 \cdot  M^{15/4} |B|^3 \cdot |B|^{32/13} \ll_M |B|^{10-1/26} \,.
\]
	Hence for an absolute constant $c>0$ one has $M\gg  |B|^c$ and thus we obtain \eqref{f:sigma_new}. 
This completes the proof. 
$\hfill\Box$
\end{proof}

\bigskip

At the end of
our article 
we formulate a hypothesis in the spirit of paper \cite{BR-NZ}.

{\bf Conjecture.} Let $Q\subset \R$ be a combinatorial cube. 
Then for any  integer $m$ there is an integer $n=n(m)$ such that $|Q^n| \gg |Q|^m$. 
Is it true that the polynomial growth takes place?



\bigskip

\noindent{Steklov Mathematical Institute,\\
	ul. Gubkina, 8, Moscow, Russia, 119991}
\\
and
\\
IITP RAS,  \\
Bolshoy Karetny per. 19, Moscow, Russia, 127994\\
and 
\\
MIPT, \\ 
Institutskii per. 9, Dolgoprudnii, Russia, 141701\\
{\tt ilya.shkredov@gmail.com}

\end{document}